\documentclass[10pt,bezier]{article}
\usepackage{amsmath,amssymb,amsfonts}

\newtheorem{prethm}{{\bf  Theorem}}

\newenvironment{thm}{\begin{prethm}{\hspace{-0.5
               em}{\bf .}}}{\end{prethm}}
\newtheorem{prepro}{{\bf  Theorem}}

\newtheorem{precor}{{\bf  Corollary}}

\newtheorem{prepos}{{\bf  Preposition}}

\newtheorem{preconj}{{\bf  Conjecture}}

\newtheorem{preremark}{{\bf  Remark}}

               \newtheorem{preexample}{{\bf  Example}}

\newenvironment{example}{\begin{preexample}{\hspace{-0.5
               em}{\bf .}}}{\end{preexample}}

                      \newtheorem{preclaim}{{\bf Claim}}

\newtheorem{prelem}{{\bf  Lemma}}

\newenvironment{lem}{\begin{prelem}{\hspace{-0.5
               em}{\bf .}}}{\end{prelem}}

\newtheorem{preproof}{{\bf  Proof.}}

\newenvironment{proof}[1]{\begin{preproof}{\rm
               #1}\hfill{$\Box$}}{\end{preproof}}

\title{\large \bf The $f$-Chromatic Index of Claw-free Graphs Whose $f$-Core\\ is $2$-regular
\thanks {2010{ \it Mathematics Subject Classification}: 05C15,
05C38.
 }}

\author{{\normalsize
{\sc S. Akbari${}^{\mathsf{b}, \mathsf{c}}$},\,
 {\sc M. Chavooshi${}^{\mathsf{c}}$},\,
 {\sc M. Ghanbari${}^{\mathsf{c}}$},\,
  {\sc R. Manaviyat${}^{\mathsf{a}}$}}\vspace{3mm}
\\{\footnotesize{${}^{\mathsf{a}}$\it
Department of   Mathematics, Payame Noor University, Tehran, Iran
}} {\footnotesize{}}\\{\footnotesize{${}^{\mathsf{b}}$\it
Department of
Mathematical Sciences, Sharif University of Technology, Tehran,
Iran
}}
{\footnotesize{}}\\{\footnotesize{${}^{\mathsf{c}}$\it
School of
Mathematics, Institute for Research in Fundamental Sciences (IPM),}}{\footnotesize{}}\\{\footnotesize{${}^{\mathsf{}}$\it
P.O. Box: 19395-5746,
 Tehran, Iran.}}
\thanks{{\it E-mail addresses}:  $\mathsf{s\_akbari@sharif.edu}$,
$\mathsf{chavooshi.m@gmail.com}$,
$\mathsf{marghanbari@gmail.com}$,
$\mathsf{ra\_manaviyat@yahoo.com}$. } }

\date{}
\begin{document}

\maketitle

\begin{abstract}
{\small
Let $G$ be a graph and  $f:V(G)\rightarrow \mathbb{N}$ be a function.
 An $f$-coloring of a graph $G$ is an edge coloring such that each color appears at each vertex $v\in V(G)$ at
most $f (v)$ times. The minimum number of colors needed
to $f$-color $G$ is called the $f$-chromatic index of $G$ and
is denoted by $\chi'_{f}(G)$. It was shown that for every graph $G$, $\Delta_{f}(G)\le \chi'_{f}(G)\le \Delta_{f}(G)+1$, where $\Delta_{f}(G)=\max_{v\in V(G)} \lceil \frac{d_G(v)}{f(v)} \rceil$. A graph $G$ is  said to be $f$-Class $1$ if $\chi'_{f}(G)=\Delta_{f}(G)$, and $f$-Class $2$, otherwise. Also, $G_{\Delta_f}$ is the induced subgraph of $G$ on $\{v\in V(G):\,\frac{d_G(v)}{f(v)}=\Delta_{f}(G)\}$.
In this paper, we show that if $G$ is a connected graph with $\Delta(G_{\Delta_f})\leq 2$ and $G$ has an edge cut of
 size at most $\Delta_f(G) -2 $
 which is a matching or  a star, then $G$ is $f$-Class $1$.
Also, we prove that if  $G$ is a connected graph and every connected component of $G_{\Delta_f}$ is
a unicyclic graph  or a tree and $G_{\Delta_f}$ is not $2$-regular, then $G$ is $f$-Class $1$.
Moreover, we show that except one  graph, every connected claw-free graph $G$ whose $f$-core is $2$-regular with a vertex $v$ such that  $f(v)\neq 1$ is $f$-Class $1$.  }

\end{abstract}

\section{Introduction}
 All graphs considered in this paper are simple and finite.
Let $G$ be a graph. The number of vertices of $G$ is called the order of $G$ and is denoted by $|G|$. Also, $V(G)$
and $E(G)$ denote the vertex set and the edge set of $G$,
respectively. The degree of a vertex $v$ in $G$ is denoted by $d_G(v)$ and $N_G(v)$  denotes the set of all vertices adjacent to $v$.
For a subgraph $H$ of $G$, $d_H(v)=|N_G(v)\cap V(H)|$.
Also, let $\Delta(G)$ and $\delta(G)$ denote the maximum degree
and the minimum degree of $G$, respectively.
A {\it star graph} is a graph containing a vertex  adjacent to all other vertices and with no extra edges.
 A {\it matching} in a graph is a set of pairwise non-adjacent
edges.  An {\it edge cut} is
a set of edges whose removal produces a subgraph with more connected
components than the original graph.  Moreover, a graph is {\it $k$-edge connected} if the minimum number of edges whose removal
would disconnect the graph  is at least $k$.  We mean $G\setminus H$, the
induced subgraph on $V(G)\setminus V(H)$. For two subsets $S$ and $T$ of $V(G)$, where $S\cap T =\emptyset$,
{\it $e_G(S, T)$} denotes the number of edges with one end in $S$ and
other end in $T$.
For a subset $X\subseteq V(G)$, we denote the induced subgraph of $G$
on $X$ by $\langle X\rangle$. A graph $G$ is called a {\it unicyclic}
graph if it is connected and contains exactly one cycle.

 A {\it  $k$-edge coloring} of a graph
$G$ is a function $f: E(G)\longrightarrow L$, where $|L| = k$ and
$f(e_1)\neq f(e_2)$, for every two adjacent edges $e_1,e_2$ of $G$.
The
minimum number of colors needed to color the edges of
$G$ properly is called the {\it chromatic index} of $G$ and is denoted by $\chi'(G)$.
Vizing \cite{9} proved that $\Delta(G) \leq \chi'(G) \leq \Delta(G)+1$,  for any
graph $G$. A graph $G$ is said to be
{\it Class $1$} if $\chi'(G) = \Delta(G)$ and {\it Class $2$} if
$\chi'(G) = \Delta(G) + 1$. A graph $G$ is called  {\it critical} if $G$ is connected, Class $2$ and $\chi'(G\setminus e) < \chi'(G)$, for
every edge $e \in E(G)$. Also, $G_\Delta$ is the induced subgraph on all vertices of degree $\Delta(G)$.

For  a function $f$ which assigns a positive integer $f(v)$ to each
vertex $v \in V(G)$, an {\it $f$-coloring} of $G$ is an edge coloring of $G$ such that each vertex $v$ has
at most $f (v)$ edges colored with the same color. The minimum number of colors needed
to $f$-color $G$ is called the {\it $f$-chromatic index} of $G$, and denoted by {\it $\chi'_{f}(G)$}.
For a graph $G$, if $f (v) = 1$ for all $v \in V(G)$, then the $f$-coloring of $G$ is reduced to the proper edge coloring of $G$. Let $\Delta_{f}(G)=\max_{v\in V(G)} \lceil \frac{d_G(v)}{f(v)} \rceil$. A graph $G$ is said to be {\it $f$-Class $1$} if $\chi'_{f}(G)=\Delta_{f}(G)$ and  {\it $f$-Class $2$}, otherwise. Also, we
say that $G$ has a $\Delta_f(G)$-coloring if $G$ is $f$-Class $1$. A vertex $v$ is called an {\it $f$-maximum vertex} if $d_G(v) =
f(v)\Delta_{f}(G)$.  A graph $G$ is called  {\it $f$-critical} if $G$ is connected, $f$-Class $2$ and $\chi'_f(G\setminus e) < \chi'_f(G)$, for
every edge $e \in E(G)$.
The {\it $f$-core} of a graph $G$ is the induced subgraph of
$G$ on the  $f$-maximum vertices and denoted by $G_{\Delta_f}$. The following example introduces an $f$-Class $1$ graph.

\begin{example}
Let $G$ be a graph shown in the following  figure such that $f(v_1)=f(v_2)=2$ and $f(v_i)=1$, for $i=3,\ldots,7$.
It is easy to see that  $\Delta_f(G)=2$, $G_{\Delta_f}=K_3$ and $G$ is $f$-Class $1$.
\end{example}

\vspace{1.5cm}
 \includegraphics{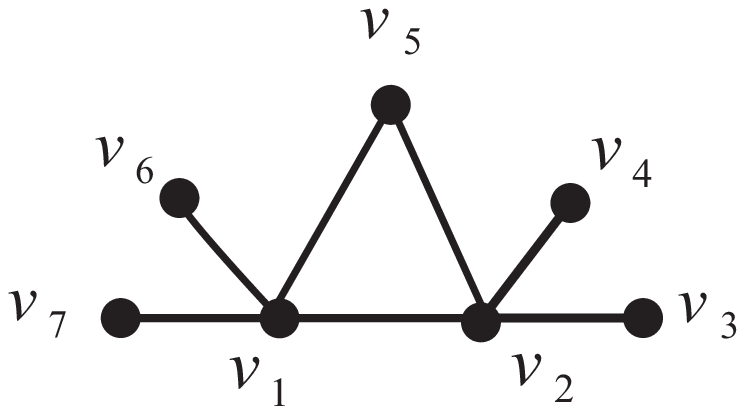} \vspace{1cm} $$\textmd{Figure\,\,$1$: An $f$-Class $1$ graph}$$

%

In {\rm \cite{3}}, Hakimi and Kariv obtained the following three results.

\begin{thm}\label{vizing}
{Let $G$ be a graph. Then $$\Delta_{f}(G) \leq \chi'_{f}(G) \leq {\rm  max}_{v\in V(G)} \lceil \frac{d_G(v)+1}{f(v)} \rceil \leq \Delta_{f}(G)+1.$$}\end{thm}

\begin{thm}
{Let $G$ be a bipartite graph. Then $G$ is $f$-Class $1$.
}\end{thm}

\begin{thm}
{Let $G$ be a graph and $f(v)$ be even, for all $v\in V(G)$. Then $G$ is $f$-Class $1$.
}\end{thm}

\begin{thm}{\rm\cite{37}}\label{tohi}
{Let $G$ be a graph. If $f(v)\nmid d(v)$ for all $v\in V(G_{\Delta_f})$, then $G$ is $f$-Class $1$.
}\end{thm}

Following result due to Zhang, Wang and  Liu gave a series of sufficient conditions for a
graph $G$ to be  $f$-Class $1$ based on the $f$-core of $G$.

\begin{thm}{\rm\cite{38}}\label{forest}
{Let $G$ be a  graph. If $G_{\Delta_f}$ is a forest, then $G$ is  $f$-Class $1$.}
\end{thm}

In {\rm \cite{critical}}, some properties of $f$-critical graphs are given. In the following, we review one of them.

\begin{thm}\label{2}
{For every vertex  $v$ of an $f$-critical graph $G$,
 $v$ is adjacent to at least $2f (v)$ $f$-maximum vertices and
$G$ contains at least three $f$-maximum vertices.}
\end{thm}

There are some theorems in proper edge coloring of graphs as follows:

\begin{thm}\label{item1}{\rm\cite{10}}
{ Let $G$ be a connected Class $2$ graph with $\Delta
(G_{\Delta})\leq 2$. Then:

  {\rm $1.$} $G$ is critical;

  {\rm $2.$} $\delta (G_{\Delta})=2$;

  {\rm $3.$} $\delta (G)= \Delta (G)-1$, unless G is an odd cycle.

}
\end{thm}

 \begin{thm}\label{cut}{\rm\cite{akb}}
{Let $G$ be a connected graph and $\Delta(G_{\Delta})\leq 2$.
Suppose that $G$ has an edge cut of
 size at most $\Delta(G) -2 $
 which is a matching or a star. Then $G$ is Class $1$. }
\end{thm}

\begin{thm}\label{edgecoloringunicyclic}{\rm\cite{akb}}
{Let $G$ be a connected graph. If every connected component of $G_\Delta$ is
a unicyclic graph  or a tree and $G_\Delta$ is not $2$-regular, then $G$ is Class $1$. }
\end{thm}

In \cite{self}, Theorem  \ref{item1} and Theorem \ref{cut} in a case that the edge cut is a matching were generalized to $f$-colorings.

\begin{thm}\label{itemf}{\rm\cite{self}}
{ Let $G$ be a connected  $f$-Class $2$  graph with $\Delta
(G_{\Delta_f})\leq 2$. Then the followings hold:

  {\rm $1.$} $G$ is  $f$-critical;

  {\rm $2.$} $G_{\Delta_f}$ is  $2$-regular;

  {\rm $3.$} $d_G(v)= f(v)\Delta_{f}(G)-1$, for every $v \in V(G)\setminus V(G_{\Delta_f})$.
 }
\end{thm}

\begin{thm}\label{fcutmatching}{\rm\cite{self}}
{Let $G$ be a connected graph and $\Delta(G_{\Delta_f})\leq 2$.
Suppose that $G$ has an edge cut of
 size at most $\Delta_f(G) -2 $
 which is a matching. Then $G$ is   $f$-Class $1$ and $G$ has a $\Delta_f(G)$-coloring in which the edges of the edge cut have different colors.}
\end{thm}

In this paper, we prove Theorems \ref{cut} and   \ref{edgecoloringunicyclic} in $f$-coloring of graphs.
Moreover, we show that except one  graph, every connected claw-free graph $G$ whose $f$-core is $2$-regular with a vertex $v$ such that  $f(v)\neq 1$ is $f$-Class $1$.

\section{Results}

In this section, we generalize Theorems \ref{cut},  \ref{edgecoloringunicyclic}  and   we obtain some results in $f$-coloring of claw-free graphs whose $f$-core is $2$-regular.
To see this, first we want to prove that if a  connected graph $G$ with  $\Delta(G_{\Delta_f})\leq 2$ has
an edge cut of
 size at most $\Delta_f(G) -2 $
 which is a matching or  a star, then $G$ is $f$-Class $1$. To do this, first  we need a lemma which is proved in  {\rm\cite{self}}.

\begin{lem}\label{lemstar}{\rm\cite{self}}
{Let $G$ be a connected graph with $\Delta(G_{\Delta_f})\leq 2$.
Suppose that $F=\{u v_1,\ldots, u v_k\}$,  $k\leq \Delta_f(G)-2$, is an
edge cut of $G$ and $f(u)=1$. Then $G$ is   $f$-Class $1$.  }
\end{lem}

\begin{thm} \label{fcut}
{Let $G$ be a connected graph and $\Delta(G_{\Delta_f})\leq 2$.
Suppose that $G$ has an edge cut of
 size at most $\Delta_f(G) -2 $
 which is a matching or  a star. Then $G$ is $f$-Class $1$. }
\end{thm}

\begin{proof}
{Clearly, we can assume that $\Delta_f(G)\geq 3$. By Theorem \ref{fcutmatching}, we can assume that the edge cut is a star.
 Let $F=\{u v_1,\ldots, u v_k\}$,  $k\leq \Delta_f(G)-2$, is an
edge cut and $V(G)=X \cup Y$, $X \cap Y= \emptyset$ and every edge
of $F$ has one end point in $X$ and other end point in $Y$. Let
$G_1$ and $G_2$ be the induced subgraphs on $X$ and $Y$,
respectively.  With no loss of generality, assume that $u \in
V(G_1)$ and $v_i \in V(G_2)$, for $i=1,\ldots,k$.
 By Lemma \ref{lemstar} and Theorem \ref{fcutmatching}, we can assume that $k,f(u)\ge 2$.
 To the contrary assume that $G$ is $f$-Class $2$. Since $\Delta(G_{\Delta_f})\le 2$ by Theorem \ref{itemf},  $G$ is $f$-critical and noting that $f(u)\geq 2$, by Theorem \ref{2}, $u\not \in V(G_{\Delta_f})$. Thus by Theorem \ref{itemf}, $d_G(u)=f(u)\Delta_f(G)-1\ge 2\Delta_f(G)-1$.
 Also, note that since $G$ is $f$-critical, by Theorem \ref{2}, $|V(G_i)\cap V(G_{\Delta_f})|\ge 2$, for $i=1,2$. Let $N_{G_1}(u)=\{w_1,\ldots,w_t\}$.
Note that since $G$ is $f$-Class $2$ and $\Delta_f(G)\geq 3$, by Theorem \ref{fcutmatching}, $G$ has no  cut edge. Thus there exists a component $D$ of $G_1\setminus \{u\}$ such that $|N_D(u)|\ge 2$. With no loss of generality, let $w_1,w_t\in V(D)$.
     Add three new vertices $x$, $y$  and $z$ to $G\setminus\{u\}$. Then join  $x,y$ and $z$  to $\{w_1,\ldots,w_{\Delta_f(G)-k}\}$,  $\{w_{\Delta_f(G)-k+1},\ldots,w_t\}$
     and $\{v_1,\ldots,v_k\}$, respectively. Let $H=\langle (V(G_1)\setminus \{u\})\cup \{x,y\}\rangle$
     and $K=\langle V(G_2)\cup \{z\}\rangle$. Let $f':V(H\cup K)\longrightarrow \mathbb{N}$ be a function defined by
     \begin{center}
$f'(v)=
\begin{cases}
   f(v)      &v \in V(G),\\
   1          &v\in\{x,z\},\\
   f(v)-1       & v=y.\
\end{cases}$
\end{center}

      Note that $H$ and
$K$ are connected. Moreover, $\max(\Delta_{f'}(H), \Delta_{f'}(K))\le \Delta_f(G)$, because
\begin{center}
$\frac{d(v)}{f'(v)}=
\begin{cases}
   \frac{d(v)}{f(v)}\le \Delta_f(G)      &v \in V(G),\\
   \Delta_{f}(G)-k<\Delta_f(G)          &v=x,\\
   k\le \Delta_{f}(G)-2<\Delta_f(G)      &v=z,\\
   \frac{d_G(u)-\Delta_f(G)}{f(u)-1}=\frac{f(u)\Delta_f(G)-1-\Delta_f(G)}{f(u)-1}<\Delta_f(G)       & v=y.\
\end{cases}$
\end{center}
 and since  $|V(G_i)\cap V(G_{\Delta_f})|\ge 2$, for $i=1,2$,   $\Delta_{f'}(H)= \Delta_{f'}(K)=\Delta_f(G)$.

We claim that both $H$ and  $K$ are $f'$-Class $1$. Note that if $H$ is $f'$-Class $2$,  then by Theorem \ref{itemf},
$d_H(x)=f'(x)\Delta_{f'}(H)-1=\Delta_f(G)-1$, but $d_{H}(x)=\Delta_{f}(G)-k\leq \Delta_f(G)-2$, a
contradiction.  So, there exists an $f'$-coloring
$\phi$ of $H$ by colors $\{1,\ldots,\Delta_{f'}(H)\}$. Similarly, there is an $f'$-coloring
$\theta$ of $K$ by colors $\{1,\ldots,\Delta_{f'}(K)\}$ and the claim is proved.

By a suitable permutation of  colors, one may
assume that $\{\phi(xw_1),\ldots,\phi(xw_{\Delta_{f}(G)-k}),\theta(zv_1),\ldots,\theta(zv_k)\}$ are distinct.
 Now, define an $f$-coloring $c: E(G)
\longrightarrow \{1, \ldots, \Delta_{f}(G)\} $ as follows:
\begin{center}
$\begin{cases}
c(e)= \phi(e)         &{\rm for\,\,every}\,\,\,e\in E(G_1\setminus\{u\}),\\
c(e')= \theta(e')            &{\rm for\,\,every}\,\,\,e'\in E(G_2),\\
c(uv_i)=\theta(zv_i)      &{\rm for}\,\,\,i=1,\ldots,k,\\
  c(uw_i)=\phi(xw_i)      &{\rm for}\,\,\,i=1,\ldots,\Delta_{f}(G)-k,\\
   c(uw_i)=\phi(yw_i)     &{\rm for}\,\,\,i=\Delta_{f}(G)-k+1,\ldots,t.
   \end{cases}$
\end{center}
 This implies that $G$ is $f$-Class $1$ which is a contradiction and the proof is complete.
 }
\end{proof}

Now, we want to prove another result in $f$-coloring of graphs which classify some of  $f$-Class $1$ graphs. To do this, we need the following lemma.

\begin{lem}\label{schf}{\rm\cite{schf}}
{Let $C$ denote the set of colors available to color the edges of a simple graph $G$. Suppose that $e = uv$
is an uncolored edge in $G$, and graph $G\setminus \{e\}$ is $f$-colored with the colors in $C$. If for every neighbor $x$ of either $u$ or $v$, there exists a color $\alpha_x$ which appears at most $f(x)-1$ times, then there exists an $f$-coloring of $G$ using colors of $C$.
}
\end{lem}

\begin{thm} \label{unicyclic}
{Let $G$ be a connected graph. If every connected component of $G_{\Delta_f}$ is
a unicyclic graph  or a tree and $G_{\Delta_f}$ is not $2$-regular, then $G$ is $f$-Class $1$. }
\end{thm}

\begin{proof}
{  First suppose that $\Delta(G_{\Delta_f})\leq 2$. Toward a contradiction,
assume that $G$ is $f$-Class $2$. By Theorem
  \ref{itemf}, $G_{\Delta_f}$ is $2$-regular, which is a contradiction. So, one may suppose that $\Delta(G_{\Delta_f})\geq
  3$.     Now, the proof is by induction on $m=|E(G_{\Delta_f})|$.
     Since $\Delta(G_{\Delta_f})\geq   3$, we have $m\ge 3$. First assume that $m=3$.  Then $G_{\Delta_f}=K_{1,3}$. Consider $H=G\setminus \{e\}$, where $e=uv\in E(G_{\Delta_f})$ and $d_{G_{\Delta_f}}(v)=1$.
 Consider the graph $H$ with function $f$.  We want to show that $H$ is $f$-Class $1$.
 If $H$ is connected, then $H_{\Delta_f}$ is the union of  two isolated vertices. Then by Theorem \ref{itemf}, $H$ is $f$-Class $1$. Now, assume that $H$ is not connected. Let $P$ and $Q$ be two connected components of $H$ such that
$u\in V(P)$ and $v\in V(Q)$. Note that $\Delta_f(P)=\Delta_f(G)$. Since $\delta(P_{\Delta_f})=0$, by Theorem \ref{itemf},  $P$ is $f$-Class $1$.
Now,  if $\Delta_f(Q)<\Delta_f(G)$, then  by Theorem \ref{vizing}, $Q$  has an $f$-coloring with colors $\{1,\ldots,\Delta_f(G)\}$.
Also, if $\Delta_f(Q)=\Delta_f(G)$, then clearly $Q_{\Delta_f}=\emptyset$ and by Theorem \ref{tohi}, $Q$ is $f$-Class $1$. Also,
  Now,
since for every $x \in N_G(v)\setminus \{u\}$, we have  $x\not \in V(G_{\Delta_f})$, there exists a color $\alpha_x$ which appears at most $f(x)-1$ times in $x$ and so
by Lemma \ref{schf}, $G$ is $f$-Class $1$ and we are done.

 Let $G$ be a graph  and $t=|E(G_{\Delta_f})|$. Assume that  the assertion holds for all graphs with $m < t$. Consider $H=G\setminus \{e\}$, where $e=uv$ is one of edges of $G_{\Delta_f}$  such that  $d_{G_{\Delta_f}}(v)=1$ and $d_{G_{\Delta_f}}(u)\ge 2$.
 Consider the graph $H$ with function $f$. We would like  to show that $H$ is $f$-Class $1$. Two cases may occur.

  First assume that $H$ is connected. If $\Delta(H_{\Delta_f})\ge 3$, then by the induction hypothesis we are done. If $\Delta(H_{\Delta_f})\le 2$ and $H_{\Delta_f}$ is not $2$-regular, then by Theorem \ref{itemf}, $H$ is $f$-Class $1$. Thus assume that $H_{\Delta_f}$ is $2$-regular. Then it is not hard to see that $G_{\Delta_f}$ is a disjoint union of some cycles and the graph shown in the following figure:

  \includegraphics{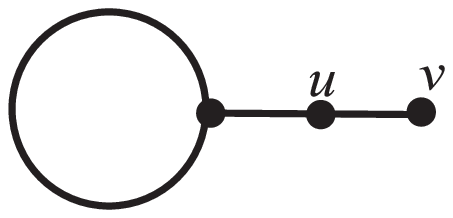}\vspace{2cm} $$\textmd{Figure 2: A part of $G_{\Delta_f}$} $$
Now, by Theorem \ref{itemf}, $H$ is $f$-critical and so by Theorem \ref{2}, $u$ should have at least two neighbors in $H_{\Delta_f}$,  a contradiction.

Next assume that $H$ is not connected. Let $P$ and $Q$ be two connected components of $H$ such that
$u\in V(P)$ and $v\in V(Q)$.
Clearly, $\Delta_f(P)=\Delta_f(G)$. If $\Delta(P_{\Delta_f})\ge 3$, then by the induction hypothesis, $P$ is $f$-Class $1$. If $\Delta(P_{\Delta_f})\le 2$ and $P_{\Delta_f}$ is not $2$-regular, then by Theorem \ref{itemf}, $P$ is $f$-Class $1$. Thus assume that $P_{\Delta_f}$ is $2$-regular. Then it is not hard to see that $G_{\Delta_f}$ is the disjoint union of some unicycles, trees  and the graph shown in the Figure $2$.
Now, by Theorem \ref{itemf}, $P$ is $f$-critical and so by Theorem \ref{2}, $u$ should have at least two neighbors in $P_{\Delta_f}$,  a contradiction and $P$ is $f$-Class $1$. Now, if $\Delta_f(Q)<\Delta_f(G)$, then  by Theorem \ref{vizing}, $Q$  has an $f$-coloring with colors $\{1,\ldots,\Delta_f(G)\}$.
So, assume that  $\Delta_f(Q)=\Delta_f(G)$. Now, if $Q_{\Delta_f}=\emptyset$, then  by Theorem \ref{tohi}, $Q$ is $f$-Class $1$. Otherwise, similar to the argument about $P$, $Q$ is $f$-Class $1$.  Now,
since for every $x \in N_G(v)\setminus \{u\}$, we have  $x\not \in V(G_{\Delta_f})$, there exists a color $\alpha_x$ which appears at most $f(x)-1$ times in $x$ and so
by Lemma \ref{schf}, $G$ is $f$-Class $1$ and we are done.
}
\end{proof}

\begin{thm}\label{clawf1}
{Let $G$ be a connected claw-free graph with $\Delta(G_{\Delta_f})\le 2$.
If there exists a vertex $v\in V(G)$ such that $f(v)\neq 1$ and $G\neq W$, where $W$ is the graph shown in the following figure, then $G$ is $f$-Class $1$.
}
\end{thm}

 \vspace{10cm}
\includegraphics{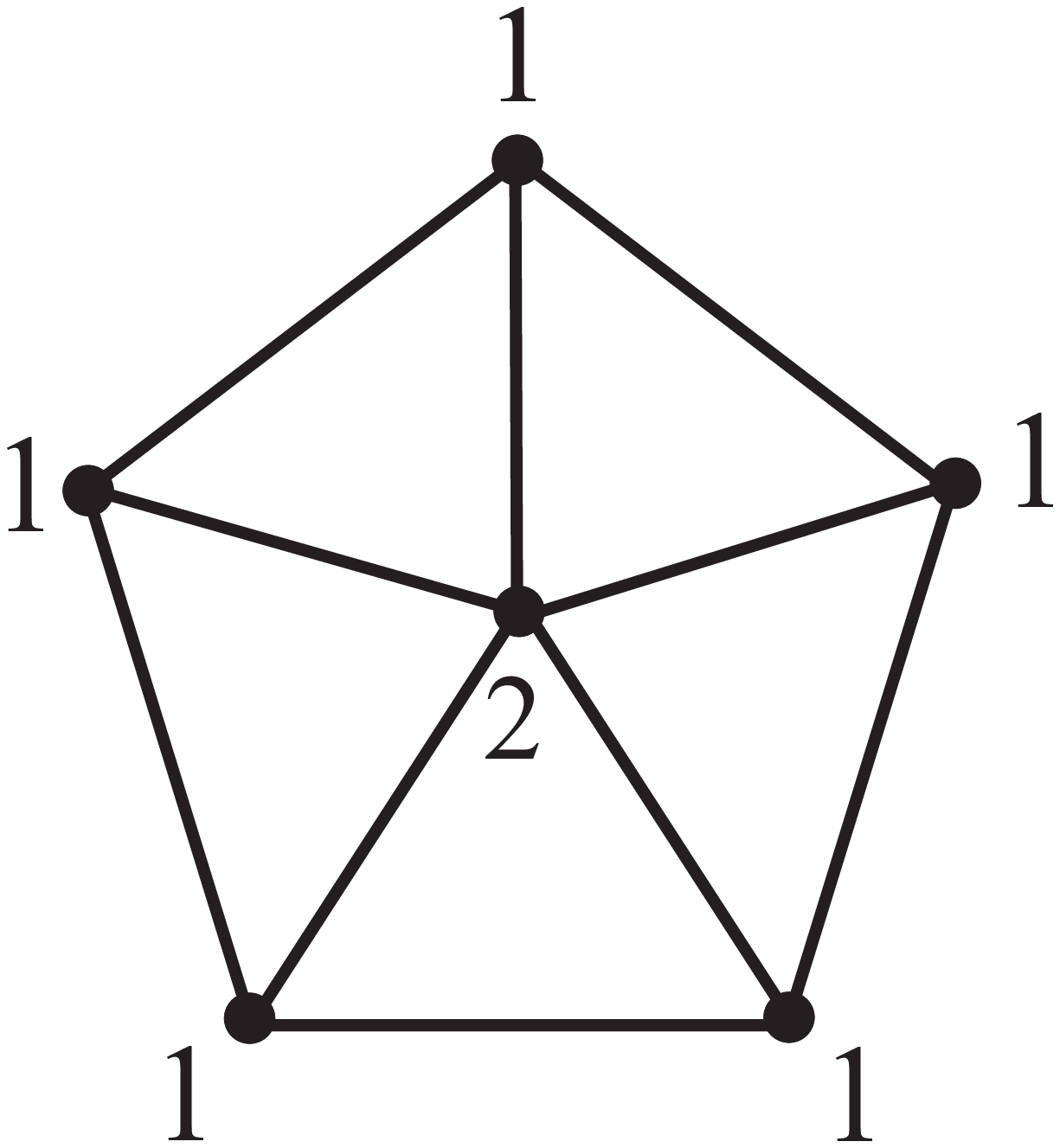} \vspace{30cm} $$\textmd {Figure\,\,$3$: The graph $W$ (The value of each vertex $z$ denotes $f(z)$)}$$

\begin{proof}
{To the contrary assume that $G$ is $f$-Class $2$.
Then by Theorem \ref{itemf}, $G$ is $f$-critical and $G_{\Delta_f}$ is $2$-regular. Now, by Theorem \ref{2}, $f(u)=1$, for every $u \in V(G_{\Delta_f})$ and so by the definition we have
\begin{equation}\label{uf}
\text {$d_G(u)=\Delta_f(G)$, for every $u \in V(G_{\Delta_f})$.}
\end{equation}
Note that, if $\Delta_f(G)=2$, then since $G$ is connected and $G_{\Delta_f}$ is $2$-regular, $G=G_{\Delta_f}$ and there is no vertex $v$ with $f(v)\neq 1$.
Thus we can assume that
\begin{equation}\label{deltage3}
\Delta_f(G)\ge 3.
\end{equation}
Let $H=G \setminus G_{\Delta_f}$. Now, if $|N_{G_{\Delta_f}}(x)|\ge 7$, for some $x\in V(H)$, then clearly there exists an independent set of size $3$ in $N_{G_{\Delta_f}}(x)$ which implies that $G$ has a claw, a contradiction. Thus we have
\begin{equation}\label{6}
 \text { $|N_{G_{\Delta_f}}(x)|\leq 6$, for every $x\in V(H)$.}
\end{equation}

Now, to prove the theorem, first we need the following claim:

\noindent{$\mathbf{Claim\,\, 1.}$} $f(z)\le 2$, for every $z\in V(G)$.

 \noindent{$\mathbf{Proof\,\, of\,\, Claim\,\, 1.}$} To see this by the contrary, assume that there exists a vertex $z\in V(G)$ such that $f(z)\ge 3$.  Clearly, $z\in V(H)$.
Now, by (\ref{6}) and Theorem \ref{2}, we conclude that $|N_{G_{\Delta_f}}(z)|=6$ and $f(z)=3$. Then by Theorem \ref{itemf}, $d_G(z)=3\Delta_f(G)-1$ and so
$d_{H}(z)=3\Delta_f(G)-7$. Now, we want to show that for every $w\in N_{H}(z)$, $|N_{G_{\Delta_f}}(z)\cap N_{G_{\Delta_f}}(w)|\ge 3$.
Because otherwise, there are at least $4$ vertices, say $u_1,u_2,u_3,u_4\in N_{G_{\Delta_f}}(z)$, such that $wu_i \not \in E(G)$, for $i=1,\ldots,4$. Now, since $G_{\Delta_f}$ is $2$-regular, with no loss of generality, we can assume that $u_1 u_2 \not \in E(G)$. Then $\langle u_1,u_2,w,z\rangle$ is a claw, a contradiction. Thus, we conclude that for every $w\in N_{H}(z)$, $|N_{G_{\Delta_f}}(w)\cap N_{G_{\Delta_f}}(z)|\ge 3$. This implies that
$$3(3\Delta_f(G)-7)\le e_G( N_{G_{\Delta_f}}(z),N_{H}(z))\le 6(\Delta_f(G)-3),$$
which yields that $\Delta_f(G)\le 1$, a contradiction and the claim is proved.

Now, by the assumption of theorem and Claim $1$, we can assume that there exists a vertex $v\in V(H)$ such that $f(v)=2$. Then, by Theorem \ref{itemf}, $d_G(v)=2\Delta_f(G)-1$.
Now, by Theorem \ref{2} and using  (\ref{6}), we have $4\le|N_{G_{\Delta_f}}(v)| \le 6$. Thus, three cases may occur:

$\mathbf{Case\,\, 1.}$ $|N_{G_{\Delta_f}}(v)|=4$.\\
 Let  $N_{G_{\Delta_f}}(v)=\{u_1,\ldots,u_4\}$ and $N_H(v)=\{w_1, \ldots,w_{2\Delta_f(G)-5}\}$. Since $G_{\Delta_f}$ is $2$-regular, with no loss of generality, there are two non-adjacent vertices $u_1,u_2\in N_{G_{\Delta_f}}(v)$. Since $G$ is claw-free, $u_1w_i\in E(G)$ or $u_2 w_i\in E(G)$, for $i=1,\ldots, 2\Delta_f(G)-5$. Thus,
$$2\Delta_f(G)-5\le e_G(N_H(v), \{ u_1,u_2\})\le 2(\Delta_f(G)-3),$$
 a contradiction.

$\mathbf{Case\,\, 2.}$  $|N_{G_{\Delta_f}}(v)|=5$.\\
Let $N_{G_{\Delta_f}}(v)=\{u_1,\ldots,u_5\}$ and $N_H(v)=\{w_1, \ldots,w_{2\Delta_f(G)-6}\}$.
 First note that  since $G$ is claw-free, $N_{G_{\Delta_f}}(v)$ does not contain an independent set of size $3$ and so it can be easily checked that $\langle N_{G_{\Delta_{f}}}(v)\rangle $ is one of two following graphs:

\vspace{2cm}
 \includegraphics{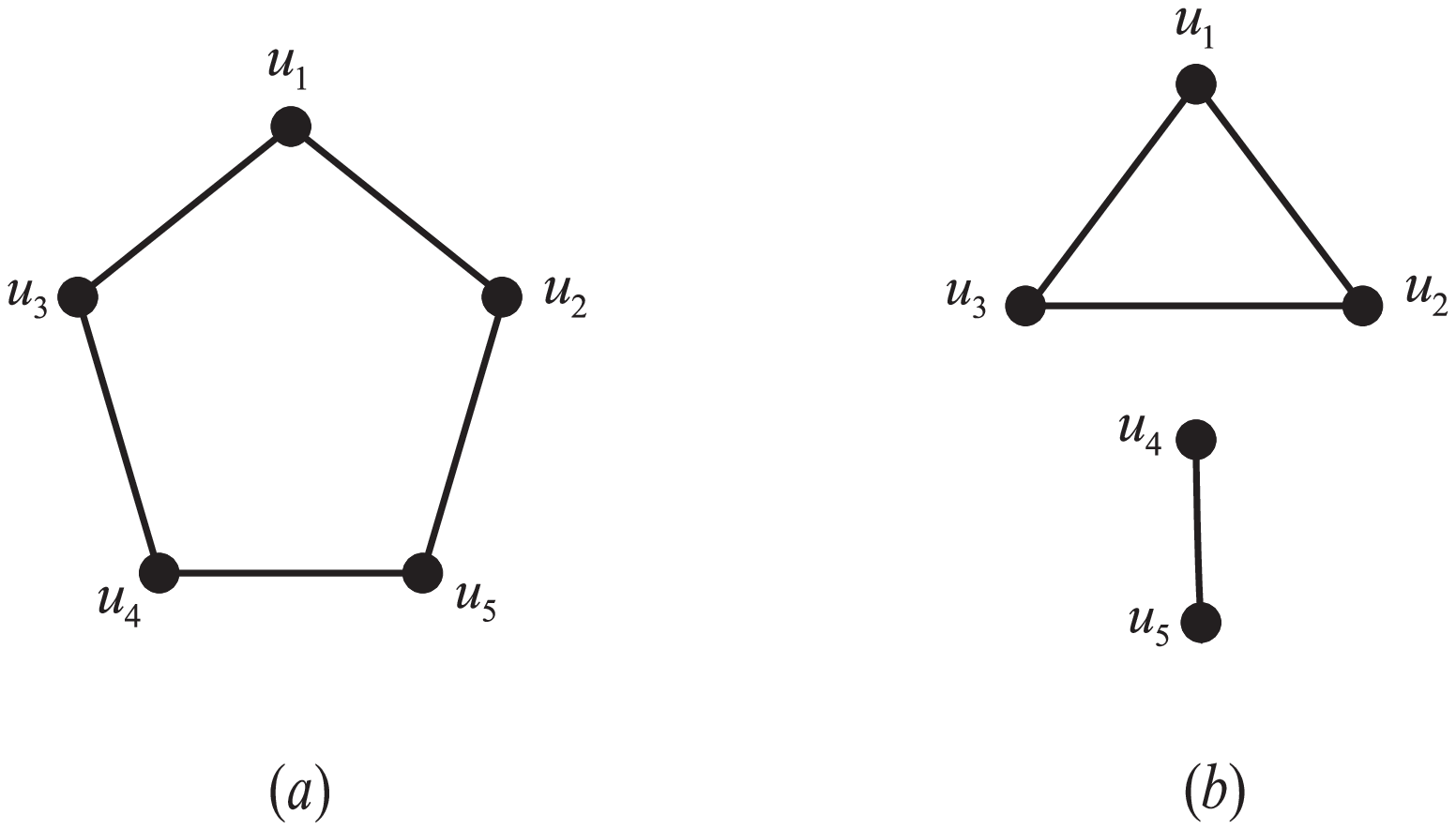} \vspace{2.5cm} $$\textmd{Figure\,\,$4$:  $\langle N_{G_{\Delta_{f}}}(v)\rangle $ when  $|\langle N_{G_{\Delta_{f}}}(v)\rangle|=5 $}$$
Three subcases may occur:

 $(i)$ $\Delta_f(G)=3$.\\
We have $d_G(v)=2\Delta_f(G)-1=5$. Now, if $\langle N_{G_{\Delta_{f}}}(v)\rangle=C_{5}$, then $G$ is the graph shown in Figure $3$, a contradiction.  Thus, assume that $\langle N_{G_{\Delta_{f}}}(v)\rangle$ is the graph shown in Figure $4(b)$. By Theorem \ref{itemf}, since $G_{\Delta_f}$ is $2$-regular, there exists  $u_{6}\in
N_{G_{\Delta_{f}}}(u_{5})\backslash \{u_{4} \}$ and $u_{7}\in N_{G_{\Delta_{f}}}(u_4)\backslash \{u_{5} \}$.  Now, we divide the proof of this subcase into two parts:

$\bullet$ $u_{6}\neq u_{7}$. Let $L=G\backslash \{v,u_{1},\ldots, u_{5}\}$. Now, add a new vertex $x$ to $L$ and join $x$ to $u_{6}$ and $u_{7}$.
Call the resultant graph $L'$.
 Let $f':V(L')\longrightarrow \mathbb{N}$ be a function defined by
     \begin{center}
$f'(z)=
\begin{cases}
   f(z)      &z \in V(L),\\
      1       & z=x.\
\end{cases}$
\end{center}
Clearly, $L'$ is a connected graph with $\Delta_{f'}(L')=\Delta_{f}(G)=3$. Note that since $d_{L'}(x)=2$, we have  $x\not\in V(L'_{\Delta_{f'}})$ and so $\delta(L'_{\Delta_{f'}})= 1$. Now, since $\Delta(L'_{\Delta_{f'}})\le  2$ and $L'_{\Delta_{f'}}$ is not $2$-regular,  by Theorem \ref{itemf}, $L'$ has an $f'$-coloring call $\theta$, with colors $\{1, 2, 3\}$. With no loss of generality, assume that $\theta(xu_{7})=1$ and $\theta(xu_{6})=2$. Now, define an $f$-coloring $c:E(G)\rightarrow \{1, 2, 3\}$ as follows.
\\Define $c(e)=\theta(e)$, for every $e\in E(L)$ and
\begin{center}
$\begin{cases}
c(u_{4}u_{7})=c(vu_{1})=c(vu_{5})=c(u_{2}u_{3})=1\\
  c(u_{5}u_{6})=c(vu_{4})=c(vu_{3})=c(u_{1}u_{2})=2\\
   c(u_{4}u_{5})=c(vu_{2})=c(u_{1}u_{3})=3.
   \end{cases}$
\end{center}

$\bullet$ $u_{6}=u_{7}$. Since $\Delta_{f}(G)=3$, $u_{6}$ has a neighbor $t$, where $t\not\in\{v,u_{1},\ldots,u_{5}\}$. Clearly, $tu_{6}$ is a cut edge for $G$ and by Theorem \ref{fcut}, $G$ is $f$-class $1$, a contradiction.

 $(ii)$ $\Delta_f(G)=4$.\\
 Clearly, $d_G(v)=2\Delta_f(G)-1=7$ and  $N_H(v)=\{w_1,w_2\}$.  Now, we divide the proof of this subcase into two parts:

$\bullet$ $\langle N_{G_{\Delta_{f}}}(v)\rangle$ is the graph shown in Figure $4(a)$.\\
 Since $G$ is claw-free, noting that $u_1u_4\not \in E(G)$, we have $u_1w_1\in E(G)$ or $u_4w_1\in E(G)$. With no loss of generality assume that $u_1w_1\in E(G)$.
  Moreover, since $\langle v,u_1,u_4,w_2\rangle$ is not a  claw and $N_G(u_1)=\{v,u_2,u_3,w_1\}$, we have $u_4w_2\in E(G)$. Similarly,
  since $\langle v,u_1,u_5,w_2\rangle$ is not a  claw and $N_G(u_1)=\{v,u_2,u_3,w_1\}$, we conclude that $u_5w_2\in E(G)$.
  Also, since $\langle v,u_2,u_4,w_1\rangle$ is not a  claw and $N_G(u_4)=\{v,u_3,u_5,w_2\}$, we obtain that $u_2w_1\in E(G)$.
  Moreover, since $\langle v,u_3,u_5,w_1\rangle$ is not a  claw and $N_G(u_5)=\{v,u_2,u_4,w_2\}$, $u_3w_1\in E(G)$.
  Now, clearly $\langle v,u_2,u_3,w_2\rangle$ is a  claw which is a contradiction.

$\bullet$  $\langle N_{G_{\Delta_{f}}}(v)\rangle$ is the graph shown in Figure $4(b)$.\\
 Similar to the previous argument, we can assume that $\{u_1w_1,u_2w_1,u_3w_1,u_4w_2,u_5w_2\}\subseteq E(G)$. Now, since $d_G(w_1)\ge 4$, $f(w_1)\ge 2$. Now, by Claim $1$ we  conclude that $f(w_1)=2$ and so by Theorem \ref{itemf}, $d_G(w_1)=7$. Assume that
 $N_G(w_1)=\{v,v_1,v_2,v_3,u_1,u_2,u_3\}$.
 Now, by Theorem \ref{2}, we have $|N_{G_{\Delta_f}}(w_1)|\ge  4$. If $|N_{G_{\Delta_f}}(w_1)|= 4$, then by Case $1$, we are done. So, we can assume that $|N_{G_{\Delta_f}}(w_1)|\ge 5$. Thus we have
 \begin{equation}\label{v1v2gdeltaf}
 v_1,v_2\in V(G_{\Delta_f}).
\end{equation}
 Also, since $\langle u_1,v_i,v_j,w_1\rangle$ is not a claw, for $i,j=1,2,3$ and $N_G(u_1)=\{v,u_2,u_3,w_1\}$, we obtain that
  \begin{equation}\label{vivjadjacent}
 \text{$\langle v_1,v_2,v_3\rangle=k_3,$}
\end{equation}
 Now, we show that
 \begin{equation}\label{w1notadjw2}
 v_3 \neq w_2
\end{equation}
 To the contrary assume that $v_3=w_2$. Then $d_G(w_2)\ge 6$ and since $w_2\not \in V(G_{\Delta_f})$, we have $f(w_2)=2$. Let $N_G(w_2)=\{v,v_1,v_2,u_4,u_5,w_1,y\}$, where $y \not \in \{u_1,u_2,u_3\}$.
 Now, since $\langle u_1,u_4,w_2,y\rangle$ and $\langle u_1,u_5,w_2,y\rangle$ are not claws, we conclude that $\langle u_4,u_5,y \rangle$ is a $K_3$ in $G_{\Delta_f}$ and so $yv_1\not \in E(G)$. Then $\langle v,v_1,w_2,y \rangle$ is a claw, a contradiction and (\ref{w1notadjw2}) holds.

 Now, consider $L=G\setminus \{v,u_1,u_2,u_3,w_1\}$. Add a new vertex $x$ to $L$ and join $x$ to $u_5,w_2,v_2,v_3$.
Call the resultant graph $L'$.
 Define $f':V(L')\longrightarrow \mathbb{N}$ be a function defined by
     \begin{center}
$f'(z)=
\begin{cases}
   f(z)      &z \in V(L),\\
      1       & z=x.\
\end{cases}$
\end{center}
  Clearly by (\ref{vivjadjacent}), $L'$ is connected and $v_1,u_4\not \in V(L'_{\Delta_{f'}})$.  Now, if $v_3 \not\in V(G_{\Delta_f})$, then clearly $\delta(L'_{\Delta_{f'}})=1$ and $\Delta(L'_{\Delta_{f'}})\le 2$ and by Theorem \ref{itemf}, $L'$ is $f'$-Class $1$. So, assume that $v_3 \in V(G_{\Delta_f})$.  Clearly $L'_{\Delta_{f'}}$ is not $2$-regular and each of the components is a unicyclic graph or a tree. Now,  by Theorem \ref{unicyclic}, $L'$ has an $f'$-coloring call $\theta$, with colors $\{1, 2, 3,4\}$. With no loss of generality, assume that $\theta(xu_5)=1$,
   $\theta(xw_2)=2$, $\theta(xv_3)=3$ and  $\theta(xv_2)=4$. Now, define an $f$-coloring $c:E(G)\rightarrow \{1, 2, 3,4\}$ as follows.

   Let $c(e)=\theta(e)$, for every $e\in E(L)$, $c(vu_5)=1, c(vw_2)=2, c(v_3w_1)=3, c(v_2w_1)=4$ and $c(vu_4)=a$, $c( v_1w_1)=b$, where $a$ and $b$ are the colors missed in coloring $\theta$ in $u_4$ and $v_1$, respectively. Let
   $\langle v,u_1,u_2,u_3,w_1\rangle$.
Now, by a suitable $f$-coloring of $\langle v,u_1,u_2,u_3,w_1\rangle$, we extend the $f'$-coloring $\theta$ of $L'$ to an $f$-coloring $c$ of $G$.
With  no loss of generality, one of the following cases may occur:

 \includegraphics{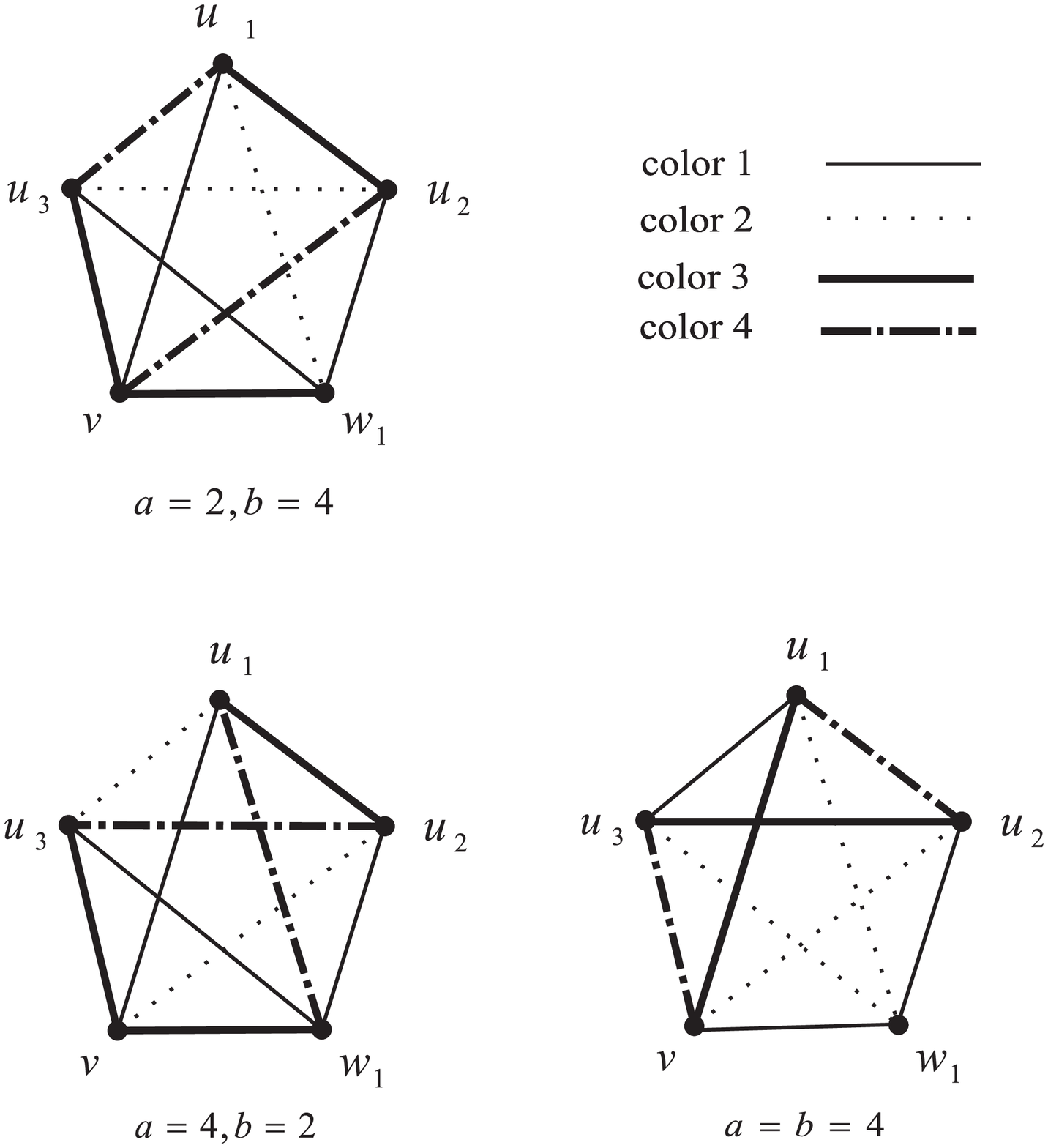} \vspace{10cm} $$\textmd{Figure\,\,$5$: $4$-edge coloring of $\langle v,u_1,u_2,u_3,w_1\rangle$}$$

%
%
%

 $\bullet$ $a=4$, $b=3$. It is easy to see that one of coloring of graphs shown in Figure $5$ works in this case.

$(iii)$ $\Delta_f(G)\ge 5$.\\
Consider $G\setminus \{v\}$. Now, add two new vertices $v_1$ and $v_2$ to $G\setminus \{v\}$, join $v_1$  to $\{u_1,w_1,\ldots,w_{\Delta_f(G)-1}\}$ and $v_2$ to $\{u_2,\ldots,u_5,w_{\Delta_f(G)},\ldots,w_{2\Delta_f(G)-6}\}$.
Call the resultant graph by $L$.
 Let $f':V(L)\longrightarrow \mathbb{N}$ be a function defined by
\begin{center}
$f'(z)=
\begin{cases}
   f(z)      &z \in V(G)\setminus \{v\},\\
      1       & z\in \{v_1,v_2\}.\
\end{cases}$
\end{center}
 It is easy to see that  $L$ is connected, $\Delta_{f'}(L)=\Delta_f(G)$ and $V(L_{\Delta_{f'}})=V(G_{\Delta_f})\cup \{v_1\}$.
 Noting that $|N_{L_{\Delta_{f'}}}(v_1)|=1$ and using Theorem \ref{unicyclic}, $L$ has an $f'$-coloring with colors $\{1,\ldots,\Delta_{f'}(L)\}$, call $\theta$.

Now, define an $f$-coloring $c: E(G)\longrightarrow \{1, \ldots, \Delta_{f}(G)\} $ as follows. Let
\begin{center}
$\begin{cases}
c(e)= \theta(e)      &{\rm for\,\,every}\,\,\,e\in E(G\setminus \{v\})\\
c(vu_1)=\theta(u_1v_1)  &{}\\
c(vu_i)=\theta(u_iv_2)      &{\rm for}\,\,\,i=2,\ldots,5\\
  c(vw_i)=\theta(v_1w_i)      &{\rm for}\,\,\,i=1,\ldots,\Delta_{f}(G)-1\\
   c(vw_i)=\theta(v_2w_i)     & {\rm for}\,\,\,i=\Delta_{f}(G),\ldots,2\Delta_{f}(G)-6.
   \end{cases}$
\end{center}
This implies that $G$ is $f$-Class $1$, a contradiction.

$\mathbf{Case\,\, 3.}$  $|N_{G_{\Delta_f}}(v)|=6$.\\
Let $N_{G_{\Delta_f}}(v)=\{u_1,\ldots,u_6\}$ and  $N_H(v)=\{w_1, \ldots,w_{2\Delta_f(G)-7}\}$. Since $G$ is claw-free, every induced subgraph of order $3$ of $\langle N_{G_{\Delta_f}}(v) \rangle$ has at least one edge. Thus  $\langle N_{G_{\Delta_f}}(v)\rangle$ is disjoint union of two $K_3$.
With no loss of generality, assume that
\begin{equation}\label{uk31}
\text {$\langle u_1,u_2,u_3\rangle \simeq \langle u_4,u_5,u_6\rangle \simeq K_3$.}
\end{equation}
 Thus, one can assume that
\begin{equation}\label{vtwok3}
\text {for every vertex $x$ with $f(x)=2$, $\langle N_{G_{\Delta_f}}(x)\rangle$ is the disjoint union of two $K_3$.}
\end{equation}
Clearly, since $d_{G}(v)=2\Delta_{f}(G)-1\geq 6$, we conclude that $\Delta_f(G)\geq 4$. Now, three cases may be considered:

$(i)$  $\Delta_f(G)=4$.\\
Clearly, $d_G(v)=2\Delta_f(G)-1=7$ and  $N_H(v)=\{w_1\}$. We claim that $|N_{G_{\Delta_f}}(v)\cap N_{G_{\Delta_f}}(w_1)|\geq 3$. Because otherwise,  $w_1u_{i_j} \not\in E(G)$, for $j=1,\ldots,4$, where $u_{i_j}\in N_{G_{\Delta_f}}(v)$. By (\ref{vtwok3}) and with no loss of generality, we can assume that $u_{i_1}u_{i_2}\not\in E(G)$. Then $\langle v, w_1, u_{i_1}, u_{i_2}\rangle$ is a claw, a contradiction.  Now, we divide the proof of this subcase into two parts:

$\bullet$ $|N_{G_{\Delta_f}}(v)\cap N_{G}(w_1)|\geq 4$. Then, $d_G(w_1)\ge 5$ and since $\Delta_f(G)=4$, we conclude that $f(w_1)\geq 2$ and by Claim $1$ we find that $f(w_1)=2$. Now, using (\ref{vtwok3}), $\langle N_{G_{\Delta_f}}(w_1)\rangle$ is disjoint union of two $K_3$. Since $|N_{G_{\Delta_f}}(w_1)\cap \{u_1,u_2,u_3\}|\geq 1$ and $| N_{G_{\Delta_f}}(w_1)\cap \{u_4,u_5,u_6\}|\geq 1$, we conclude that $ N_{G_{\Delta_f}}(w_1)=\{u_1,\ldots,u_6\}$. Then, it is easy to see that $G$ is the graph shown in the following figure
 which is colored with $\Delta_f(G)=4$ colors and the proof of this subcase is complete.\\
 \includegraphics{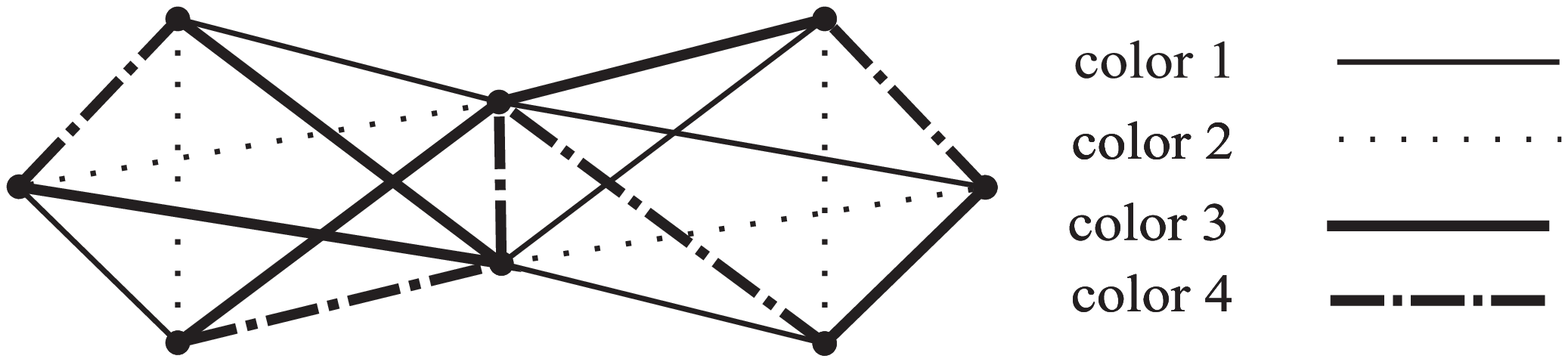} \vspace{3cm} $$\textmd{Figure\,\,$6$: An $f$-coloring of $G$ with $4$ colors }$$

$\bullet$ $|N_{G_{\Delta_f}}(v)\cap N_{G}(w_1)|=3$. Since $d_G(w_1)\geq 4$ and $w_1\not\in V(G_{\Delta_f})$, $f(w_1)=2$. Using (\ref{vtwok3}) and with no loss of generality we can assume that $ N_{G_{\Delta_f}}(w_1)\cap N_{G_{\Delta_f}}(v)= \{u_1,u_2,u_3\}$ and there are three vertices call $x_1,x_2,x_3\in N_{G_{\Delta_f}}(w_1)$ such that $\langle x_1,x_2,x_3\rangle \simeq K_3$. Consider $L=G\setminus  \{u_1,u_2,u_3,vw_1\}$. Let $f':V(L)\longrightarrow \mathbb{N}$ be a function defined by
 $f'(z)=f(z)$, for every $z \in V(L)$.
  Now, we want to prove the following claim which introduces a coloring of $L$ with some properties.

  \noindent{$\mathbf{Claim\,\, 2.}$} $L$ has an $f'$-coloring $c$ with four colors $\{1,2,3,4\}$ such that $$|\{c(w_1x_1),c(w_1x_2),c(w_1x_3),c(vu_4),c(vu_5),c(vu_6)\}|=4.$$

\noindent {$\mathbf{Proof\,\, of\,\, Claim\,\, 2.}$}  We consider two cases.

 First assume that $L$ is not connected. So, $L$ has two connected components, one of them containing $v$ and another containing $w_1$.
   It is easy to see that for every connected component $I$ of $L$, $\Delta_{f'}(I)=\Delta_{f'}(L)=\Delta_f(G)$ and so  $\Delta(I_{\Delta_{f'}})= 2$. Now, since $f'(v)=f'(w_1)=2$ and $d_{L}(v)=d_{L}(w_1)=3$, by Theorem \ref{itemf}, every component of $L$ is $f'$-class $1$.
Moreover, noting that $f'(v)=f'(w_1)=2$,  we obtain that there are at least two distinct colors appeared in the edges incident with $v$ and also with $w_1$.
   Now, by a suitable permutation of colors on these edges in one of components, Claim $2$ is proved.

 Now, assume that $L$ is connected. Consider $K=L\setminus \{w_1,x_1x_2,x_2x_3,x_1x_3\}$. Let $f'':V(K)\longrightarrow \mathbb{N}$ be a function defined by
\begin{center}
$f''(z)=
\begin{cases}
   f'(z)      &z \in V(L)\setminus \{w_1\},\\
      1       & z=v.\
\end{cases}$
\end{center}
 We want to show that $K$ is $f''$-Class $1$. It is not hard to see that every connected component of $K$ has at least one of the three vertices $\{x_1,x_2,x_3\}$.
  Let $J$ be a connected component of $K$. If $\Delta_{f''}(J)<\Delta_{f''}(K)$, then by Theorem \ref{vizing}, $J$ has an $f''$-coloring with $4$ colors.
  So, assume that  $\Delta_{f''}(J)=\Delta_{f''}(K)=4$. Now, since there exists $x_{i}\in V(J)$, for some $i\in \{1,2,3\}$ and noting that $d_J(x_{i})=1$, by Theorem \ref{itemf}, $J$ is $f''$-Class $1$ and so $K$ has an $f''$-coloring with $4$ colors $\{1,2,3,4\}$, call $\theta$. Let $N_{K}(x_1)=\{y_1\}$, $N_{K}(x_2)=\{y_2\}$ and $N_{K}(x_3)=\{y_3\}$. We want to show that
   \begin{equation}\label{9}
\text {$|\{\theta(x_1y_1),\theta(x_2y_2),\theta(x_3y_3)\}|\geq 2$.}
\end{equation}
Because otherwise, we have  $|\{\theta(x_1y_1),\theta(x_2y_2),\theta(x_3y_3)\}|=1$. Now, since for every vertex $u\in V(G)$, $f(u)\leq 2$, we conclude that  $|\{y_1,y_2,y_3\}|\geq 2$. With no loss of generality, one can suppose that $y_1$ is not adjacent to $x_2$ and $x_3$. Using (\ref{vtwok3}), we find that $f(y_1)=1$ and so $f''(y_1)=1$. Thus since $d_{K}(y_1)=\Delta_{f''}(K)-1=3$,  there is a missed color call $\alpha$ in $y_1$ different from $\theta(x_1y_1)$. One can replace $\theta(x_1y_1)$ by $\alpha$. Thus (\ref{9}) holds and  we are done.

  Now, with no loss of generality
  and noting that $f''(v)=1$, one can assume that $\theta(vu_4)=1$, $\theta(vu_5)=2$, $\theta(vu_6)=3$, $\theta(x_1y_1)=\alpha$, $\theta(x_2y_2)=\beta$ and $\theta(x_3y_3)=\gamma$.
  Now, to prove Claim $2$, it suffices to extend the $f''$-coloring of $K$  to an $f'$-coloring of $L$. To see this, in the following figures, we introduce such a suitable coloring for $\langle w_1,x_1,x_2,x_3\rangle\cup \{x_1y_1,x_2y_2,x_3y_3\}$.

 \includegraphics{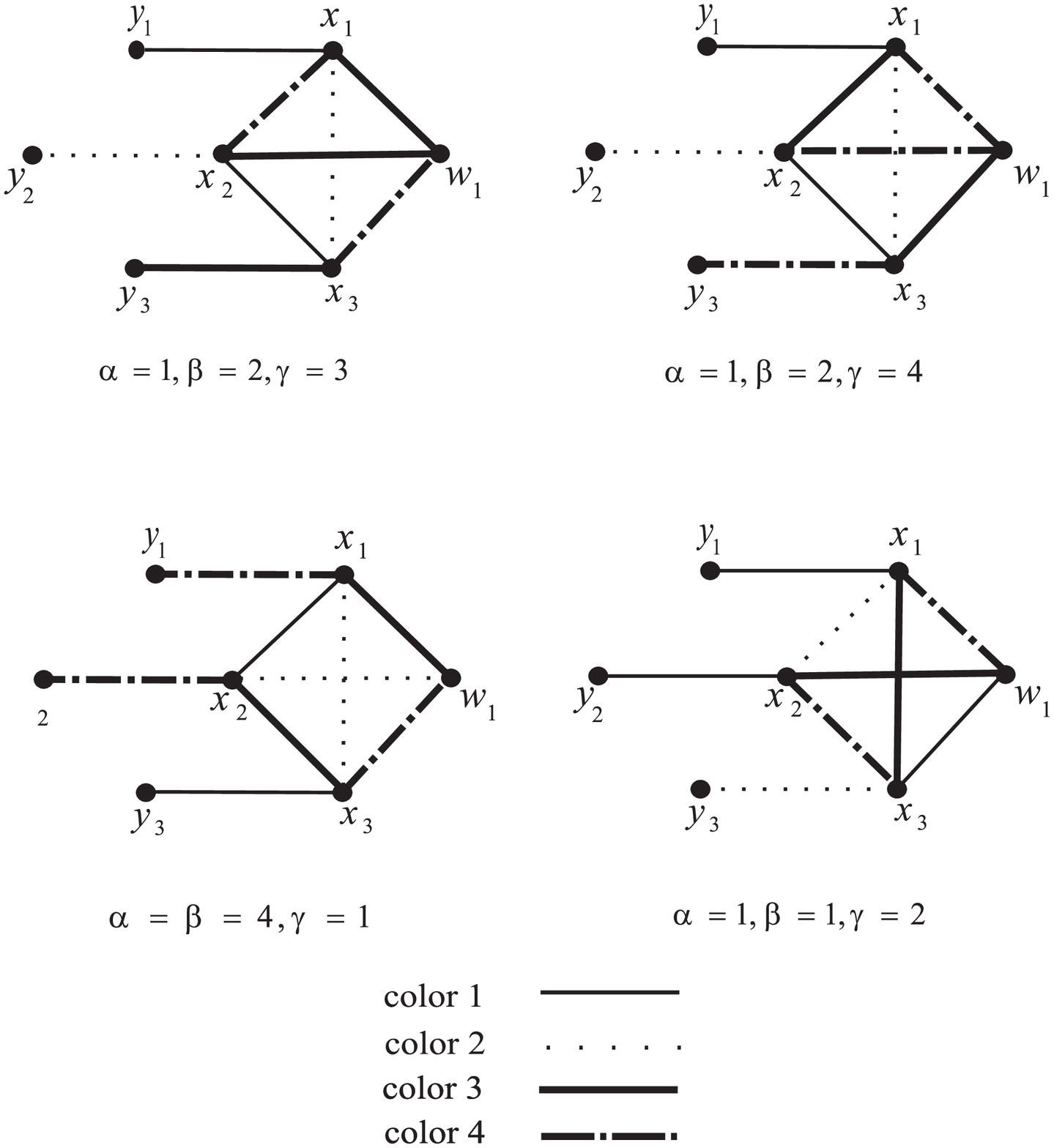} \vspace{12cm} $$\textmd{Figure\,\,$7$: A $4$-edge coloring of $\langle w_1,x_1,x_2,x_3\rangle \cup \{x_1y_1,x_2y_2,x_3y_3\}$}$$

%
%
%
%
%
%

Note that if $\alpha=\beta=1$ and $\gamma=4$, then since $f''(y_1)=1$, there is a missed color in $y_1$ different from $1$. Now, by changing color $w_1 x_1$ by this missed color, similar to one of the coloring of graphs shown in Figure $7$, we are done.

Now, we can easily color $\langle v,u_1,u_2,u_3,w_1\rangle$ by colors $\{1,2,3,4\}$ similar to one of the graphs in Figure $5$. This implies that $G$ is $f$-Class $1$ and we are done.\\

$(ii)$ $\Delta_f(G)=5$.\\
 By (\ref{uk31}), $u_1u_4\not \in E(G)$. Thus $u_1w_1\in E(G)$ or $u_4 w_1\in E(G)$.
 With no loss of generality, assume that $u_1w_1\in E(G)$.
Since two graphs $\langle v,u_1,u_4,w_2\rangle$ and $\langle v,u_1,u_4,w_3\rangle$ are not claw and $d_{G}(u_{1})=5$, with no loss of generality, we can suppose that   $u_1w_2\in E(G)$ and $u_4 w_3\in E(G)$.
Moreover, since $\langle v,u_1,u_5,w_3\rangle$ and $\langle v,u_1,u_6,w_3\rangle$ are not claw and $N_G(u_1)=\{v,u_2,u_3,w_1,w_2\}$,  we have $u_5w_3,u_6 w_3\in E(G)$.
Now, we want to show that
 \begin{equation}\label{u2wiu3wi}
\text{$u_i w_j\in E(G)$, for $i=2,3$ and $j=1,2$.}
\end{equation}
  To the contrary and with no loss of generality assume that $u_2w_1 \not \in E(G)$.
   Then since $\langle v,u_2,u_i,w_1\rangle$ is not a claw, we have $u_iw_1\in E(G)$, for $i=4,5,6$. This implies that $d_G(w_1)\ge 5$ and since $\Delta_f(G)=5$, we conclude that  $f(w_1)=2$. Now, by (\ref{vtwok3}), $u_2 w_1 \in E(G)$, a contradiction. Similarly, other items of (\ref{u2wiu3wi}) hold.

    Now, we would like  to show that $G$ is $f$-Class $1$. Two cases may occur:

$\bullet$ $w_1 w_2 \not \in E(G)$.\\
Since $\langle v,u_4,w_1,w_2\rangle$ is not a claw, with no loss of generality, $u_4w_1\in E(G)$ and so $d_G(w_1)\geq 5$, which implies that $f(w_1)=2$ and by (\ref{vtwok3}),  $u_5w_1,u_6w_1\in E(G)$. Since $d_G(w_1)=9$, there exists a vertex $z\in N_G(w_1)\setminus \{v,u_1,\ldots,u_6,w_3\}$ and $\langle z,u_1,u_4,w_1\rangle$ is a claw, a contradiction and the proof of this case is complete.

$\bullet$ $w_1 w_2 \in E(G)$.\\
Clearly, $\langle v,u_1,u_2,u_3,w_1,w_2\rangle \simeq K_6$ and so
\begin{equation}\label{vtook6}
\text{for every vertex $v$ with $f(v)=2$, $v$  is contained in a $K_6$.}
\end{equation}

Note that  since $d_G(w_i)\ge 5$ and $w_i\not \in V(G_{\Delta_f})$, by  Claim $1$ we conclude that $f(w_i)=2$, for $i=1,2$.
Let $P$ be the induced subgraph on the union of vertices of all $K_6$ in $G$.
First note that three vertices of each $K_6$ have degree $5$ in $G$. This implies that every two $K_6$ have at most three vertices in common.
 Also, every two $K_6$ have not one vertex in common, because otherwise there exists a vertex of degree $10$ in $G$. On the other hand, every two $K_6$ have not three vertices in common, because otherwise there exists a vertex $v\in V(P)$ such that $d_P(v)=8$ and it is not hard to see that $v$ is a center of a claw in $G$, a contradiction. Thus, the vertex set of every two $K_6$ have empty intersection or they have exactly two vertices in common. Hence each connected component of $P$ is one of the following figures.
\vspace{2cm}
 \includegraphics{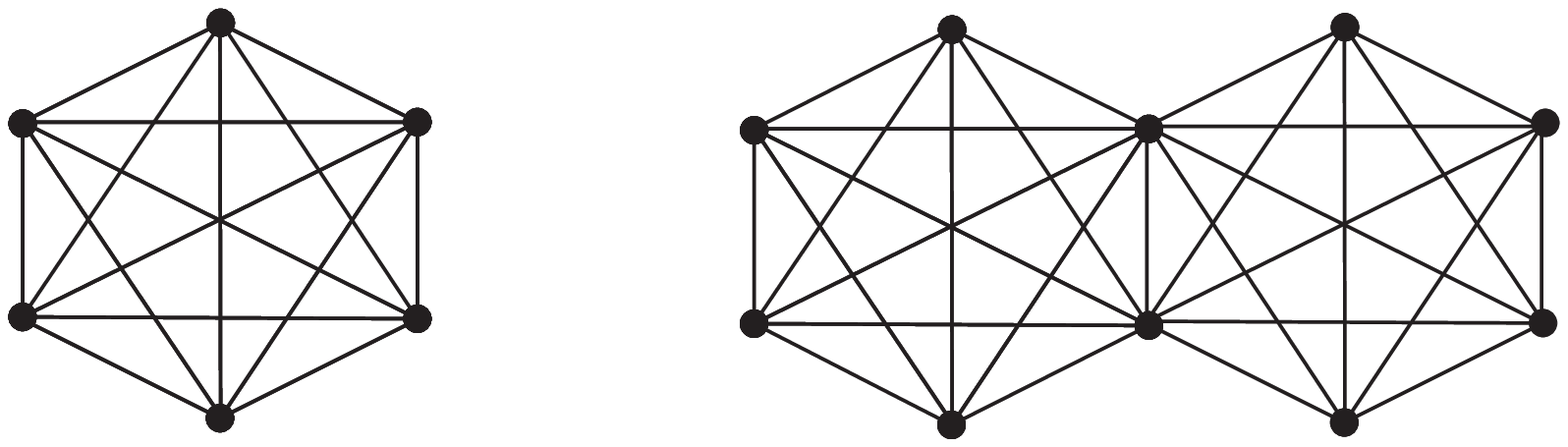} \vspace{2cm} $$\textmd{Figure\,\,$8$: Every component of the graph $P$}$$

Define $f':V(P)\longrightarrow \mathbb{N}$ as follows:
\begin{center}
$f'(z)=
\begin{cases}
   1       & if\,\,\,d_P(z)=5\\
      2      &if\,\,\,d_P(z)=9.\
\end{cases}
$
\end{center}
 It is not hard to see that $P$ has an  $f'$-coloring with colors $\{1,\ldots,5\}$.

Now, let $L=G \setminus E(P)$. We would like to prove the following claim.

\noindent{$\mathbf{Claim\,\, 3.}$} $\chi'(L)=5$.

If the claim is proved, then we color all edges of $L$ and $P$ by $5$ colors to obtain an $f$-coloring of $G$. Since for every vertex $v$ which are incident to some edges in $L$ and $P$, we have $f(v)=2$, we find an $f$-coloring of $G$ using $5$ colors.


 \noindent{$\mathbf{Proof\,\, of\,\, Claim\,\, 3.}$}
 Clearly,  the maximum degree of each connected component of $L$ is at most $5$. If the maximum degree is less than 5, then by Vizing's Theorem we are done.  Now, let $I$ be a connected component of $L$ such that $\Delta(I)=5$. Note that $V(I_{\Delta}) \subseteq V(G_{\Delta_f})$ and $\Delta(I_{\Delta}) \leq 2$. Note that since $G$ is connected, there exists a vertex $x\in V(I)\cap V(P)$ and so $d_I(x)\ge 1$. Since $\delta(P)=5$ and $d_G(x)= 9$, we conclude that $d_I(x)=4$. This implies that $f(x)=2$ and by (\ref{vtwok3}), it is not hard to see that  $\mid N_I(x)\cap V(G_{\Delta_f})\mid=3$ and so there exists a vertex $y\in N_I(x)$ such that $d_I(y)=4$. Let $N_I(x)\cap V(G_{\Delta_f})=\{u,u',u''\}$. Obviously,  since $G$ is claw-free, $yu,yu',yu''\in E(I)$.

 Let $J=I \setminus \{x,y,uu',uu'',u'u''\}$. We show that $J$ has a $5$-edge coloring.
 If $\Delta(J)\le 4$, then by Vizing's Theorem, $J$ has a $5$-edge coloring. Thus assume that $\Delta(J)=5$ and so  $\Delta(J_{\Delta})\le 2$ and $d_J(u)=d_J(u')=d_J(u'')=1$. Hence by Vizing's Theorem and Theorem \ref{item1}, every connected component of $J$ has a $5$-edge coloring. Let $N_J(u)=\{z\}, N_J(u')=\{z'\}$ and $N_J(u'')=\{z''\}$. We claim that  there exists a $5$-edge coloring of $J$ in which color of edges $uz,u'z'$ and $u''z''$ are distinct. To see this, if
$z=z'=z''$, then we are done.
 If $z\neq z'=z''$ and color of edges $uz,u'z'$ are the same and different from color of the edge $u''z''$, then since $d_J(z')=4$, we conclude that there exists a missed color in $z'$ which is different from the color of $u'z'$ and $u''z''$. Now, by substituting this missed color with the color of $u'z'$,  we  are done. Now, assume that $z, z'$ and $ z''$ are distinct.  Then, remove three vertices $u,u',u''$ of $J$. Also,   add a new vertex $s$, join $s$ to the vertices $z,z',z''$ and call the resultant graph by  $K$.
  Now, since $\Delta(K)=5$, $\Delta(K_{\Delta})\le 2$ and $\delta(K)=3$, by Theorem \ref{item1}, $K$ has a $5$-edge coloring. Now,  by a suitable extending this $5$-edge coloring to a $5$-edge coloring of $J$, we conclude that there exists a $5$-edge coloring of $J$ such that three distinct colors appear in edges $uz,u'z'$ and $u''z''$.

  Now, we want to extend the $5$-edge coloring of $J$ to a $5$-edge coloring of $I$ to complete the proof of Claim $3$. To see this, we show that there exists a $5$-edge coloring for $Q=\langle u,u',u'',x,y \rangle$
  such that three missed colors in $u,u'$ and $u''$ are distinct. Add a new vertex $q$ to $Q$ and join $q$ to $u,u',u''$ and call the resultant graph by $R$. Clearly, $R$ is the subgraph of $K_6$ and so $\chi'(R)=5$. Now,  Claim $3$ is proved.

$(iii)$  $\Delta_f(G)\ge 6$.\\
Consider $G\setminus \{v\}$, add two new vertices $v_1,v_2$ to $G\setminus \{v\}$, join $v_1,v_2$ to $\{u_1,w_1,\ldots,w_{\Delta_f(G)-1}\}$ and $\{u_2,\ldots,u_6,w_{\Delta_f(G)},\ldots,w_{2\Delta_f(G)-7}\}$, respectively.
Call the resultant graph $L$.
 Let $f':V(L)\longrightarrow \mathbb{N}$ be a function defined by
\begin{center}
$f'(v)=
\begin{cases}
   f(v)      &v \in V(G)\setminus \{v,v_1,v_2\},\\
      1       & v\in \{v_1,v_2\}.\
\end{cases}$
\end{center}
It is easy to see that  $L$ is connected, $\Delta_{f'}(L)=\Delta_f(G)$ and $V(L_{\Delta_{f'}})=V(G_{\Delta_f})\cup \{v_1\}$.
 Noting that $|N_{L_{\Delta_{f'}}}(v_1)|=1$. Now, by Theorem \ref{unicyclic}, $L$ has an $f'$-coloring with colors $\{1,\ldots,\Delta_{f'}(L)\}$, call $\theta$.

Now, define an $f$-coloring $c: E(G)\longrightarrow \{1, \ldots, \Delta_{f}(G)\} $ as follows. Let
\begin{center}
$\begin{cases}
c(e)= \theta(e) &{\rm for\,\,every}\,\,\,e\in E(G) \cap E(G')\\
c(u_1v)=\theta(u_1v_1)\\
c(u_iv)=\theta(u_iv_2)      &{\rm for}\,\,\,i=2,\ldots,6\\
  c(vw_i)=\theta(v_1w_i)      &{\rm for}\,\,\,i=1,\ldots,\Delta_{f}(G)-1\\
   c(vw_i)=\theta(v_2w_i)     & {\rm for}\,\,\,i=\Delta_{f}(G),\ldots,2\Delta_{f}(G)-7.
   \end{cases}$
\end{center}
Thus $G$ is $f$-Class $1$, a contradiction and the proof of the theorem is complete.
}
\end{proof}

%
%
%
%
%
%
%
%
%

\vspace{5mm} \centerline{\bf\large Acknowledgments}\vspace{5mm}

The authors are grateful to the referee
for useful comments.
 The first and third authors are indebted to the School of Mathematics,
Institute for Research in Fundamental Sciences (IPM) for support.
The research of the first author and third one were in part supported by a grant
from IPM (No. 92050212) and (No. 92050014), respectively.

\end{document}